\documentclass[12pt]{amsart}
\usepackage[margin=1.1in]{geometry}
\usepackage{yufei}
\usepackage{mathrsfs}
\numberwithin{equation}{section}
\title{The Pop-stack-sorting Operator on Tamari Lattices}
\date{December 2021}
\author{Letong Hong}
\address{Department of Mathematics, Massachusetts Institute of Technology, Cambridge, MA 02139}
\email{clhong@mit.edu}

\begin{document}
\begin{abstract}
Motivated by the pop-stack-sorting map on the symmetric groups, Defant defined an operator $\mathsf{Pop}_M : M \to M$ for each complete meet-semilattice $M$ by
$$\mathsf{Pop}_M(x)=\bigwedge(\{y\in M: y\lessdot x\}\cup \{x\}).$$
This paper concerns the dynamics of $\mathsf{Pop}_{\mathrm{Tam}_n}$, where $\mathrm{Tam}_n$ is the $n$-th Tamari lattice.

We say an element $x\in \mathrm{Tam}_n$ is $t$-$\mathsf{Pop}$-sortable if $\mathsf{Pop}_M^t (x)$ is the minimal element and we let
$h_t(n)$ denote the number of $t$-$\mathsf{Pop}$-sortable elements in $\mathrm{Tam}_n$.
We find an explicit formula for the generating function $\sum_{n\ge 1}h_t(n)z^n$ and verify Defant's conjecture that it is rational. We furthermore prove that the size of the image of $\mathsf{Pop}_{\mathrm{Tam}_n}$ is the Motzkin number $M_n$, settling a conjecture of Defant and Williams.
\end{abstract}
\maketitle
\section{Introduction}

Building on Knuth's stack-sorting algorithm \cite{Knuth}, West's ground-breaking work on stack-sorting map on symmetric groups \cite{West} inspired subsequent studies, including the reverse-stack-sorting map \cite{Dukes} and the pop-stack-sorting map \cite{AN}. Recently, there has been considerable attention by combinatorialists on the pop-stack sorting map \cite{ABB, ABH, CG, EG, PS}. For each complete meet-semilattice $M$, Defant defined an operator $\mathsf{Pop}_M$ that agrees with the pop-stack-sorting map when $M$ is the weak order on $S_n$ \cite{Defant_tamari}. It is defined so that $\mathsf{Pop}_M$ sends an element to the meet of itself and all elements that it covers. By definition, $M$'s minimal element $\hat{0}$ stays the same when $\mathsf{Pop}_M$ is applied. We say an element $x$ is \emph{$t$-$\mathsf{Pop}$-sortable} if $\mathsf{Pop}_M^t(x)=\hat{0}$.

Pudwell and Smith \cite{PS} enumerated the number of $2$-$\mathsf{Pop}$-sortable elements in $S_n$ under the weak order. Claesson and Gu\dh mundsson \cite{CG} proved that for each fixed nonnegative integer $t$, the generating function that counts $t$-$\mathsf{Pop}$-sortable elements in $S_n$ is rational. Defant \cite{Defant_coxeter} established the analogous rationality result for the generating functions of $t$-$\mathsf{Pop}$-sortable elements of type $B$ and type $\widetilde{A}$ weak orders.

Introduced in 1962, the $n$-th Tamari lattice $\mathrm{Tam}_n$ consists of semilength-$n$ Dyck paths (lattice paths from $(0,0)$ to $(n,n)$ above the diagonal $y=x$) \cite{Tamari}; its partial order will be defined in Section \ref{terminology}. There are generalizations of the definition, most notably the $m$-Tamari lattices by Bergeron and Pr\'eville-Ratelle \cite{BP} and the $\nu$-Tamari lattices introduced by Pr\'eville-Ratelle and Viennot \cite{PV}. Fundamental in algebraic combinatorics \cite{MPS}, the $n$-th Tamari lattice $\mathrm{Tam}_n$ is also isomorphic to $\mathrm{Av}_n(312)$, the lattice of $312$-pattern-avoiding permutations under the weak order of $S_n$ \cite{BW}. 

In this paper, we study the $\mathsf{Pop}$ operator on Tamari lattices. Let $h_t(n)$ be the number of $t$-$\mathsf{Pop}$-sortable elements in $\mathrm{Tam}_n$. A part of a conjecture by Defant \cite{Defant_tamari} is that for every fixed $t$, the generating function $\sum_{n\ge 1} h_t(n)z^n$ is rational. We confirm this statement by giving the exact formula of the generating function:
\begin{theorem}\label{Catalan}
Let $h_t(n)$ denote the number of $t$-$\mathsf{Pop}$-sortable Dyck paths in the $n$-th Tamari lattice $\mathrm{Tam}_n$. Then $$\sum_{n\ge 1} h_t(n)z^n=\frac{z}{1-2z-\sum_{j=2}^{t}C_{j-1}z^{j}},$$ where $C_j$ are the Catalan numbers.
\end{theorem}

Moreover, settling a conjecture in Defant and Williams's paper (Conjecture 11.2 (2) in \cite{DW}), we have the following theorem:
\begin{theorem}\label{Motzkin}
Define $\mathsf{Pop}(L; q)=\sum_{b\in \mathsf{Pop}_L(L)}q^{|\mathscr{U}_L(b)|}$, where $\mathscr{U}_L(b)$ is the set of elements of $L$ that cover $b$. Then we have
$$\mathsf{Pop}(\mathrm{Tam}_{n+1}; q) = \sum_{k=0}^n\frac{1}{k+1}\binom{2k}{k}\binom{n}{2k}q^{n-k},$$ where the coefficients form OEIS sequence  \cite{OEIS} A055151.

In particular, when $q=1$, we have that $$|\mathsf{Pop}_{\mathrm{Tam}_n}(\mathrm{Tam}_n)|=M_{n-1},$$ where $M_n$ is the $n$-th Motzkin number (OEIS sequence \cite{OEIS} A001006).
\end{theorem}

Additional motivation for studying the size of the image of $\mathsf{Pop}_{\mathrm{Tam}_n}$ comes from a theorem by Defant and Williams (Theorem 9.13 in \cite{DW}). In that theorem, they proved that $|X_n|= \{y\in \mathrm{Tam}_n \mid \mathsf{Row}(y)\le y\}$, where $\mathsf{Row}$ is the rowmotion operator on $\mathrm{Tam}_n$ (which is equivalent to the Kreweras complement operator on noncrossing partitions \cite{DS}). They also showed that $|X_n|$ is the number of independent dominating sets in a certain graph associated with $\mathrm{Tam}_n$ called its \emph{Galois graph}.

The paper is organized as follows. In Section \ref{terminology} we give the necessary definitions. In Section \ref{proof_result1}
and Section \ref{proof_result2} we prove \cref{Catalan} and \cref{Motzkin}. 
\section*{Acknowledgements}
The research was conducted at the 2021 University of Minnesota Duluth REU (NSF--DMS Grant 1949884 and NSA Grant H98230-20-1-0009) and fully supported by the generosity of the CYAN Mathematics Undergraduate Activities Fund. The author is deeply thankful to Professor Joseph Gallian for his long-lasting efforts and care in running the wonderful program and to Colin Defant for proposing the project and his dedicated mentorship. The author is also grateful to Qiuyu Ren and Daniel Zhu for discussions/editing comments and to Kenny Lau and Alec Sun for programming assistance.
\section{Definitions}\label{terminology}
\subsection{Lattice basics and the $\mathsf{Pop}$ operator.}
\begin{definition}
A \emph{meet-semilattice} is a poset $M$ such that any two elements $x,y\in M$ have a greatest lower bound (which is called their \emph{meet}, denoted by $x\wedge y$). 
A \emph{lattice} $L$ is a meet-semilattice such that any two elements $x,y\in L$ also have a least upper bound (which is called their \emph{join}, denoted by $x\vee y$). A meet-semilattice is \emph{complete} if every nonempty subset $A\subset M$ has a meet.\\
Given $x,y\in M$, we say that $y$ is \emph{covered} by $x$ (denoted $y\lessdot x$) if $y<x$ and no $z\in M$ satisfies $y<z<x$. 
\end{definition}
In this paper we only consider finite meet-semilattices, each of which has a unique minimal element $\hat{0}$. They are automatically complete.
\begin{definition}[\cite{Defant_tamari}]
Let $M$ be a complete meet-semilattice. Define the \emph{semilattice pop-stack-sorting operator} $\mathsf{Pop}_M:M\to M$ by $$\mathsf{Pop}_M(x)=\bigwedge(\{y\in M: y\lessdot x\}\cup \{x\}).$$
\end{definition}

\begin{definition}
    We say an element $x$ of a complete meet-semilattice $M$ is \emph{$t$-$\mathsf{Pop}$-sortable} if $\mathsf{Pop}^t(x)=\hat{0}$.
\end{definition}




\subsection{Generalized Tamari lattices.}
In this paper, a lattice path is a finite planar path that starts from the origin and at each step travels either up/$\rm{N}: (0,1)$ or right/$\rm{E}: (1,0)$. 
\begin{definition}
The \emph{horizontal distance} of a point $p$ with respect to a lattice path $\nu$ is the maximum number of east steps one can take starting from $p$ before being strictly to the right of $\nu$.
\end{definition}

\begin{definition}[\cite{PV}]
Let $\nu$ be a lattice path from $(0,0)$ to $(\ell -n,n)$. The \emph{generalized $\nu$-} \emph{Tamari lattice} $\mathrm{Tam}(\nu)$ is defined as follows: 
\begin{enumerate}
\item elements of $\mathrm{Tam}(\nu)$ are lattice paths $\mu$ from $(0,0)$ to $(\ell-n,n)$ that are weakly above $\nu$;
\item the partial order of $\mathrm{Tam}(\nu)$ is given by the covering relation: $\mu\lessdot\mu'$ if $\mu'$ is obtained by shifting a subpath $D$ of $\mu$ by $1$ unit to the left, where $D$ satisfies (i) it is preceded by $\mathrm{E}$; (ii) its first step is $\mathrm{N}$; (iii) its endpoints $p,p'$ are of the same horizontal distance to $\nu$ and there is no point between them with the same horizontal distance to $\nu$ as $p$. In other words, $\mu\lessdot \mu'$ if for such subpath $D$, $\mu=X\mathrm{E}DY$ and $\mu'=XD\mathrm{E}Y$.
\end{enumerate} 
\end{definition}

\begin{figure}[h]
\centering
\includegraphics[scale=0.7]{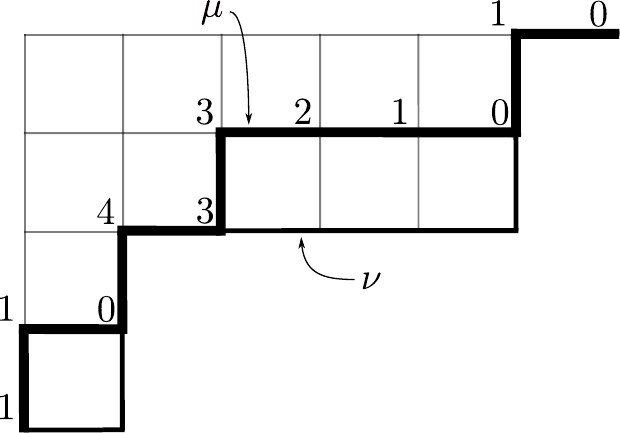}\cite{Defant_tamari}
\caption{Lattice path $\mu =$ NENENEEENE is in $\rm{Tam}(\nu)$ where $\nu =$ ENNEEEENNE.\quad Each point on $\mu$ is labeled with its horizontal distance.}\label{figure1}
\end{figure}
\begin{definition}
When $\nu=(\mathrm{NE})^n$, the lattice $\mathrm{Tam}(\nu)$ is the $n$-th \emph{Tamari lattice} $\mathrm{Tam}_n$ consisting of the \emph{Dyck paths}. It is well-known that $|\mathrm{Tam}_n|$ is the $n$-th Catalan number $C_n$.
\end{definition}

\section{Proof of \cref{Catalan}}\label{proof_result1}
\subsection{Preliminaries: the $\nu$-bracket vector.}
\begin{definition}\label{bracketvector}
Let $\mathbf{b}(\nu)=(b_0(\nu),b_1(\nu),\ldots, b_{\ell}(\nu))$ be the vector denoting the heights at each step of the lattice path $\nu$. Let the \emph{fixed position} $f_k$ denote the largest index such that $b_{f_k}(\nu)=k$. We say that an integer vector $\vec{\mathsf{b}}=(\mathsf{b}_0,\mathsf{b}_1,\ldots, \mathsf{b}_{\ell})$ is a \emph{$\nu$-bracket vector}, denoted as $\vec{\mathsf{b}}\in \rm{Vec}(\nu)$, if \begin{enumerate}

\item $\mathsf{b}_{f_k}=k$ for all $k=0,\ldots, n$.
\item $b_i(\nu)
\le \mathsf{b}_i \le n$ for all $0\le i\le \ell$.
\item If $\mathsf{b}_i =k$, then $\mathsf{b}_j \le k$ for all $i+1\le j\le f_k$.
\end{enumerate}
The partial order of $\rm{Vec}(\nu)$ is defined as follows: we say $(\mathsf{b}_0,\mathsf{b}_1,\ldots, \mathsf{b}_{\ell})\le (\mathsf{b}'_0,\mathsf{b}'_1,\ldots, \mathsf{b}'_{\ell})$ if $\mathsf{b}_i\le \mathsf{b}'_i$  for all $i$.
\end{definition}
\begin{remark}
An equivalent interpretation of (3) is that $\vec{\mathsf{b}}$ is 121-pattern-avoiding. These conditions also imply the sequence $\{\mathsf{b}_i\}_{f_{k-1}+1}^{f_k}$ is non-increasing for all $k=0,\ldots, n$.
\end{remark}

\begin{definition}
Let $\mu\in \mathrm{Tam}(\nu)$ be a path from $(0,0)$ to $(\ell-n,n)$. We define $\mathbf{b}(\mu)=(b_0(\mu),b_1(\mu),$ $\ldots,b_{\ell}(\mu))$ its \emph{associated vector} as follows:
make $(\ell+1)$ empty slots; traverse $\mu$, and when arriving at a new grid point, write its height $k$ at the rightmost available slot among those that are weakly to the left of index $f_k$.
\end{definition}
\begin{remark}
We alert the readers that the notation of the vector $\mathbf{b}(\mu)$ does not reflect its dependence on the fixed lattice path $\nu$. 
\end{remark}
\begin{example}\label{associatedvec}
We use $\mu=$ NENENEEENE and $\nu=$ ENNEEEENNE as in Figure \ref{figure1}. The fixed positions are $f_0=1$, $f_1=2$, $f_2=7$, $f_3=8$, and $f_4= 10$. Then we create 11 empty slots and construct the associated vector $\mathbf{b}(\mu)$ as follows: \begin{equation*}\begin{aligned}&\ \ (\underline{\ \  \ },\underline{\ 0 \  },\underline{\ \  \ },\underline{\ \  \  },\underline{\ \  \ },\underline{\ \  \  },\underline{\ \  \  },\underline{\ \  \  },\underline{\ \ \  },\underline{\ \  \  },\underline{\ \  \  })\to (\underline{\ 1  \  },\underline{\ 0\ },\underline{\ 1 \  },\underline{\ \  \ },\underline{\ \  \  },\underline{\ \  \  },\underline{\ \  \  },\underline{\ \  \  },\underline{\ \  \ },\underline{\ \  \ },\underline{\ \  \ }) \\\to\ & (\underline{\ 1  \  },\underline{\ 0\ },\underline{\ 1 \  },\underline{\ \  \ },\underline{\ \  \  },\underline{\ \  \  },\underline{\ 2  \  },\underline{\ 2 \  },\underline{\ \  \ },\underline{\ \  \ },\underline{\ \  \ }) \to (\underline{\ 1  \  },\underline{\ 0\ },\underline{\ 1 \  },\underline{\ 3\ },\underline{\ 3 \  },\underline{\ 3 \  },\underline{\ 2 \  },\underline{\ 2  \  },\underline{\ 3  \ },\underline{\ \  \ },\underline{\ \  \ }) \\\to\ & (\underline{\ 1  \  },\underline{\ 0\ },\underline{\ 1 \  },\underline{\ 3\ },\underline{\ 3 \  },\underline{\ 3 \  },\underline{\ 2 \  },\underline{\ 2  \  },\underline{\ 3  \ },\underline{\ 4 \ },\underline{\ 4  \ }). \end{aligned}\end{equation*}
\end{example}
\begin{theorem}[\cite{CPS}]\label{bij}
The map $\mathbf{b} : \mathrm{Tam}(\nu) \to \mathrm{Vec}(\nu)$ is an order-preserving bijection. Furthermore, for any paths $\mu, \mu' \in \mathrm{Tam}(\nu)$, we have $\mathbf{b}(\mu\wedge \mu') = \min(\mathbf{b}(\mu), \mathbf{b}(\mu'))$ the term-wise minimum vector.
\end{theorem}
\begin{notation}We define the followings.
\begin{enumerate}
    \item $\Delta(\mu):= \{ i\mid i<\ell \text{ and } b_i(\mu)>b_{i+1}(\mu)\}.$ \vspace{1.5mm}
    \item $\eta_i(\mu):= \begin{cases*}
\max\{x\in [b_i(\nu),b_i(\mu)-1] \mid b_j(\mu) \le x,\ \forall j \in [i+1,f_x]\}& if $i\in \Delta(\mu)$, \\
b_i(\mu) & if $i\not \in \Delta(\mu)$.
\end{cases*}$\vspace{1.5mm}

\item $\mathbf{b}_{\downarrow}^i(\mu):= (b_0(\mu),\ldots,b_{i-1}(\mu), \eta_i(\mu), \ldots,b_{\ell}(\mu)) $.
\end{enumerate}
\end{notation}
\begin{example}
Again we use $\mu=$ NENENEEENE as in Figure \ref{figure1} and by \cref{associatedvec} we have that $\mathbf{b}(\mu)=(1,0,1,3,3,3,2,2,3,4,4)$. Hence, $\Delta(\mu)=\{0,5\}$, $\eta_0(\mu)= 0$, and $\eta_5(\mu)= 2$.
\end{example}
\begin{proposition}[\cite{Defant_tamari}]\label{popeffect}
We have that$$ \mathbf{b}(\mathsf{Pop}_{\mathrm{Tam}(\nu)}(\mu))=(\eta_0(\mu),\eta_1(\mu),\ldots, \eta_{\ell}(\mu)).$$
\end{proposition}

\begin{corollary}[\cite{Defant_tamari}]\label{inequality}
Suppose $\mu \in \mathrm{Tam}(\nu)$ and $f_{k-1} < i < f_k\ (0\le k\le n)$. Then $b_i(\mathsf{Pop}_{\mathrm{Tam}(\nu)}(\mu)) \ge b_{i+1}(\mu)$.
\end{corollary}
We use the assumptions for a lattice path $\nu$ from above. Let $\nu^{\#}$ be the path obtained from $\nu$ by deleting its first $f_0+1$ steps. Let $\mathsf{b}^{\#}$ be the vector obtained from $\mathsf{b}$ by deleting its first $f_0+1$ entries and subtracting $1$ from all remaining entries. We call this action the \emph{hash} map. Let $\mu^{\#}$ be the unique element in $\mathrm{Tam}(\nu^{\#})$ whose associated vector is $\mathbf{b}(\mu)^{\#}$.
\begin{corollary}\label{hash}
If $\mu\in \mathrm{Tam}(\nu)$ is $t$-$\mathsf{Pop}$-sortable, then so is $\mu^{\#}\in \mathrm{Tam}(\nu^{\#})$.
\end{corollary}
\begin{proof}
This directly follows from the fact that $\eta_i(\mu)$ is determined only by $b_j(\mu)$ for $j\ge i$.
\end{proof}
\subsection{Proof of the result}
Let $H_t(z)=\sum_{n\ge 1} h_t(n)z^n$, the generating function in \cref{Catalan}.
Let $\widetilde{H}_t(z)$ be the truncated polynomial $\sum_{n= 1}^{t-1} h_t(n)z^n$.
Let $G_t(z)=\sum_{n\ge 1}g_t(n)z^n$, where $g_t(n)$ denotes the $t$-$\mathsf{Pop}$-sortable irreducible elements in $\mathrm{Vec}(\nu)$ for $\nu=\mathrm{E(NE)}^{ n-1}$. In this case, using the notations from \cref{bracketvector}, we have $f_k=2k+1$, and $b_i(\nu)=\lfloor i/2 \rfloor$. Therefore, the restrictions are $\mathsf{b}_{2k+1}=k$, $ \mathsf{b}_{2k}\in \{k,k+1,\ldots, n\}$, and that if $\mathsf{b}_i=k$, then $\mathsf{b}_j\le k$ for all $j=i+1,\ldots, 2k+1$, i.e., no 121-pattern can appear. Finally, we note that $\mathrm{Vec}(\mathrm{E(NE)}^{ n-1})\cong \mathrm{Vec}(\mathrm{(NE)}^n)\cong \mathrm{Tam}_n$. 
\begin{definition}
We say $\vec{\mathsf{b}}=(\mathsf{b}_0, \mathsf{b}_1, \ldots, \mathsf{b}_{\ell})\in \mathrm{Vec}(\nu)$ for some fixed $\nu$ is \emph{irreducible} if $\mathsf{b}_0=\mathsf{b}_{\ell}$.
\end{definition}
\begin{lemma}\label{decomposition}
Every $\nu$-bracket vector can be decomposed into irreducible $\nu_i$-bracket vectors, where $\nu$ and each $\nu_i$ are of the form $\mathrm{E(NE)}^{k-1}$. A vector is $t$-$\mathsf{Pop}$-sortable if and only if all its irreducible components are.
\end{lemma}
\begin{proof}
We first define the addition of two irreducible vectors $\vec{\mathsf{b}}\in \mathrm{Vec}(\mathrm{E(NE)}^{n_1-1})$ and $\vec{\mathsf{b}}'\in \mathrm{Vec}(\mathrm{E(NE)}^{n_2-1})$ as follows: $$\vec{\mathsf{b}}+\vec{\mathsf{b}}':=(\mathsf{b}_0,\mathsf{b}_1,\ldots, \mathsf{b}_{2n_1-1},\mathsf{b}'_0+n_1,\mathsf{b}'_1+n_1,\ldots, \mathsf{b}'_{2n_2-1}+n_1)\in \mathrm{Vec}(\mathrm{E(NE)} ^{n_1+n_2-1}).$$ To prove the first claim we induct on the length of the vector and note that it suffices to show that every bracket vector can be decomposed as the sum of an irreducible vector $\vec{\mathsf{b}}_{irr}$ and a shorter vector. Simply take $\vec{\mathsf{b}}_{irr}:=(\mathsf{b}_0,\mathsf{b}_1,\ldots, \mathsf{b}_{f_{\mathsf{b}_0}}).$ The second claim is clear.
\end{proof}
\begin{lemma}\label{decomposition_cor}Assume the notations above. Then we have
$$1
+H_t(z)=\frac{1}{1-G_t(z)}.$$
\end{lemma}
\begin{proof}
The formula is a direct corollary of \cref{decomposition}.
\end{proof}
\begin{lemma}\label{hashbij}
The hash map is a one-to-one correspondence between irreducible vectors in $\mathrm{Vec}(\mathrm{E(NE)}^{n-1})$ and bracket vectors in $\mathrm{Vec}(\mathrm{E(NE)}^{n-2})$. An irreducible vector $\vec{\mathsf{b}}$ is $t$-$\mathsf{Pop}$-sortable if and only if $\vec{\mathsf{b}}^{\#}$ is $t$-$\mathsf{Pop}$-sortable and $t \ge n-x_r+1,$ where $2x_r$ is the length of the last irreducible vector component of $\vec{\mathsf{b}}^{\#}$.
\end{lemma}
\begin{proof}
Let the irreducible vector $\vec{\mathsf{b}}\in \mathrm{Vec}(\mathrm{E(NE)}^{n-1})$ be $(n,0,u_0,u_1,\ldots, u_{2n-3})$ and $\vec{\mathsf{b}}^{\#}=(u_0-1,u_1-1,\ldots, u_{2n-3}-1)\in \mathrm{Vec}(\mathrm{E(NE)}^{n-2})$. First, it is clear that from $\vec{\mathsf{b}}^{\#}$ we can recover $\vec{\mathsf{b}}$, so the hash map is a bijection. Next, if we decompose $\vec{\mathsf{b}}^{\#}$ as the sum of some (say $r$) irreducible vectors of lengths $2x_1,\ldots, 2x_r$, respectively (corresponding to elements in $\mathrm{Vec}(\nu)$ for $\nu=(\mathrm{E(NE)}^{x_i-1}),\ 1\le i\le r$), then we can write $$\vec{\mathsf{b}}=(n,0,u_0,u_1,\ldots, u_{2n-3})=(n,0,u_0,\ldots, u_0, \ldots, n-x_r,\ldots, n-x_r, n,\ldots, n).$$ The irreducible vector $\vec{\mathsf{b}}$ being $t$-$\mathsf{Pop}$-sortable is equivalent to $\vec{\mathsf{b}}^{\#}$ being $t$-$\mathsf{Pop}$-sortable and the first entry of $\vec{\mathsf{b}}$ turning $0$ after $t$ $\mathsf{Pop}$'s. Applying $\mathsf{Pop}_{\mathrm{Vec}(\mathrm{E(NE)}^{ n-1})}$ once changes the first entry from $n$ to $n-x_r$, and each subsequent $\mathsf{Pop}_{\mathrm{Vec}(\mathrm{E(NE)}^{ n-1})}$ decreases it by $1$, hence this is then equivalent to $t\ge n-x_r+1$.
\end{proof}

\begin{lemma}\label{species}
Assume the notations above. Then we have $$G_t(z)=z\left((1+\widetilde{H}_t(z))G_t(z)+1\right).$$
\end{lemma}
\begin{proof}
This is a corollary of \cref{hashbij}.
Since the hash map's image of the middle sub-vector $(u_0-1,\ldots, u_0-1, \ldots, n-x_r-1, \ldots, n-x_r-1)\in \mathrm{Vec}(\mathrm{E(NE)}^{n-x_r-1})$ is $t$-$\mathsf{Pop}$-sortable when $n-x_r\le t-1$ and the last irreducible component starts and ends with $n$ as well, we have justified the desired expression (adding $1$ to $\widetilde{H}_t(z)$ is to account for the $r=0$ case).
\end{proof}
\begin{lemma}\label{manypops}
When $n\le t$, every path in $\mathrm{Tam}_n$ is $t$-$\mathsf{Pop}$-sortable.
\end{lemma}
\begin{proof}
Consider the path's associated vector $\vec{\mathsf{b}}\in \mathrm{Vec}(\mathrm{E(NE)}^{n-1})$. For each $0\le i\le n-1$, $\mathsf{b}_{2i}$ decreases by at least $1$ each time unless $\mathsf{b}_{2i}=\mathsf{b}_{2i+1}$. Since $n\le t$, during the $t$ applications of $\mathsf{Pop}_{\mathrm{Vec}(\mathrm{E(NE)}^{ n-1})}$ this equality will be reached. This applies to all $i$, so we obtain the minimum element's associated vector.
\end{proof}
We are now ready to prove our first main result.
\begin{proof}[Proof of \cref{Catalan}]
By \cref{manypops}, $\widetilde{H}_t(z)=\sum_{n=1}^{t-1}C_nz^n$. By \cref{species}, we have that $$G_t(z)=\frac{z}{1-\sum_{n=1}^{t}C_{n-1}z^n},$$ and substituting this into \cref{decomposition_cor}, we obtain that $$H_t(z)=\frac{G_t(z)}{1-G_t(z)}=\frac{\frac{z}{1-\sum_{n=1}^{t}C_{n-1}z^n}}{1-\frac{z}{1-\sum_{n=1}^{t}C_{n-1}z^n}}=\frac{z}{1-2z-\sum_{j=2}^tC_{j-1}z^j},$$as desired.
\end{proof}


\section{Proof of \cref{Motzkin}}\label{proof_result2}

\subsection{Preliminaries: congruence and $\mathsf{Pop}$ on subsemilattices}\label{congruencesec}
\begin{definition}
A \emph{lattice congruence} on a lattice $L$ is an equivalence relation $\equiv$ on $L$ such that if $x_1 \equiv x_2$ and $y_1 \equiv y_2$, then $x_1 \wedge y_1\equiv x_2 \wedge y_2$ and $x_1 \vee y_1\equiv x_2 \vee y_2$.
\\
For each $x \in L$, we denote by $\pi_{\downarrow}(x)$ the minimal element of the congruence class of $x$.
\end{definition}
\begin{definition}
A \emph{subsemilattice} of a lattice $L$ is a subset $M\subset L$ such that
$x \wedge y \in M
$ for all $x, y \in M$.
\end{definition}

\begin{theorem}\label{popcong}$($\cite{Defant_tamari}$)$
Let $L$ be a finite lattice. Let $\equiv$ be a lattice congruence on $L$ such that the set $M = \{\pi_{\downarrow}(x) \mid x \in L\}$ is a subsemilattice of $L$. Then for all $x \in M$, $$\mathsf{Pop}_M(x) = \pi_{\downarrow}(\mathsf{Pop}_L(x)).$$
\end{theorem}

We now provide an example that shows how the Tamari lattice can be realized as a sublattice of $S_n$.
\begin{definition}
A \emph{descent} of a permutation $x=x_1\cdots x_n$ is a pair of adjacent entries $x_{i}>x_{i+1}$. A \emph{descending run} is a maximal decreasing subsequence of $x$. The \emph{pop-stack-sorting map} is the operator on $S_n$ that reverses each descending run. 
\end{definition}
\begin{definition}
The partial order of $S_n$ defined by the following covering relation is the \emph{right weak order}: a permutation $y$ is covered by permutation $x$ if  $y$ is obtained by swapping one of $x$'s descents.\end{definition}

\begin{definition}(\cite{HNT})
Two words $u,v$ are \emph{sylvester-adjacent} if there exist $a<b<c$ and words $X,Y,Z$ such that $u=XacYbZ$ and $v=XcaYbZ$. We write $u\lhd v$.\\
Two words $u, v$ are \emph{sylvester-congruent} if there is a chain of words $u = w_0, w_1,\ldots , w_m = v$ such that $w_i$ and $w_{i+1}$ are sylvester-adjacent for all $i$ ($w_i\lhd w_{i+1}$ or $w_{i}\rhd w_{i+1}$).
\end{definition}

We say that a permutation $\pi$ is \emph{$312$-avoiding} if it has no $i<j<k$ such that $x_j<x_k<x_i$, and is \emph{$\overline{31}2$-avoiding} if it has no $i<j$ such that $x_i<x_j<x_{i-1}$.

Let $L=S_n$, and let $M=\mathrm{Av}_n(312)$ be the set of 312-avoiding permutations, both under the right weak order. It is established by Bj\"orner and Wachs \cite{BW} in their Theorem 9.6 (i) that $\mathrm{Av}_n(312)$ is a sublattice of $S_n$ and is isomorphic to the Tamari lattice $\mathrm{Tam}_n$. Reading \cite{Reading} observes that the sylvester-congruence is a lattice congruence for $S_n$ under the right weak order (note that $u \lhd v$ also implies $u\lessdot v$), and, furthermore, if we divide $S_n$ into sylvester-congruence classes, then each class has a unique 312-avoiding element. More precisely, $\mathrm{Av}_n(312)=\{\pi_{\downarrow}(x)\mid x\in S_n\}$.

A concrete description of $\pi_{\downarrow}$ is that we can compute a chain $x=y_0\rhd y_1\rhd \cdots \rhd y_m= \pi_{\downarrow}(x)$ until we must stop (one can easily show that no $XcaYbZ$ (i.e., $\overline{31}2$) pattern implies no 312 pattern), and we remark that the exact construction of the chain does not matter, that is, regardless of the order of swapping one obtains the same eventual outcome.

Therefore, \cref{popcong} tells us that $$\mathsf{Pop}_{\mathrm{Av}_n(312)}(x)=\pi_{\downarrow}(\mathsf{Pop}_{S_n}(x)).$$
This is especially helpful, given that $\mathsf{Pop}_{S_n}$ on the right hand side is equal to the easily characterized pop-stack-sorting map.
\subsection{Proof of the result}
\begin{theorem}\label{charMotzkin}
We have that $x\in X_n=\{ \mathsf{Pop}_{\mathrm{Av}_n(312)}(\mathrm{Av}_n(312))\}$ if and only if $x=x_1x_2\cdots x_n$ has no consecutive double descents and ends with $n$.
\end{theorem}
\begin{proof}
In this proof we interpret $\mathsf{Pop}$ as reversing all descending runs of a string (not required to be a permutation of $1$ to $m$), e.g., $\mathsf{Pop}(74513)=47153$, though we specify by using a subscript when it is indeed $\mathsf{Pop}_{S_m}$. We also recall the identity  $\mathsf{Pop}_{\mathrm{Av}_n(312)}(y)=\pi_{\downarrow}(\mathsf{Pop}_{S_n}(y))$ which will be used extensively.

For the ``only if'' direction, we first suppose that $x=\pi_{\downarrow}(\mathsf{Pop}_{S_n}(y))$ and we want to show that $x$ ends with $n$ and has no consecutive double descents.

It is known that every permutation in the image of $\pi_{\downarrow}$ must be 312-avoiding. We first prove that the last entry must be $n$. Wherever $n$ is located for a permutation $y$, in order for it to be 312-avoiding we must have that the segment after $n$ is decreasing. Then after the effect of $\mathsf{Pop}_{S_n}$, $n$ is put at the end of the permutation and continues to stay there when we apply $\pi_{\downarrow}$ because it is never involved as $a,b,\text{ or }c$ in any $XcaYbZ$ pattern.

Next we prove that there are no consecutive double descents. We use induction on the permutation length, and, with the base case being clear, we assume this claim holds for length $n-1$. Write $y=y_1y_2\cdots y_n$ and let $y_r=n$.

Suppose $y_n=n$. We thus know that $\mathsf{Pop}_{S_n}(y)$ ends with $n$ and it stays at the same place under the effect of $\pi_{\downarrow}$. Using the induction hypothesis, we have that $\pi_{\downarrow}(\mathsf{Pop}_{S_n}(y))$ will end with $(n-1)n$ with no double descents.

Suppose $y_{n-1}=n$. Let $y_n=k$. Let $\mathsf{Pop}_{S_n}(y)=z_1\cdots z_n$. Then $(z_{n-1},z_n)=(k,n)$ and $n$ stays at the same place throughout. We prove the following two claims: there is no 312 pattern involving $k$ after $\mathsf{Pop}_{S_n}$, and there is no 312 pattern involving $k$ at any stage in the chain of pairwise sylvester-adjacent permutations that we use to compute $\pi_{\downarrow}$. For the first claim, if there is a 312 pattern then there must be some $z_i,z_j$ such that $z_i>k>z_j$ and $i<j<n-1$. Since $\mathsf{Pop}_{S_n}$ does not change the relative position of entries in different descending runs, it must be that $z_i$ is before $z_j$ in preimage $y$. However, there is no 312 pattern initially in $y$, which is a contradiction. For the second claim, we know that $z_1\cdots z_n$ has no $z_i,z_j$ such that $z_i>k>z_j$ and $i<j<n-1$, and any swap ($XcaYbZ\to XacYbZ$) in the chain would not create such a pair as it moves a smaller element to the front of a larger element.

Therefore, we can delete $k$ and $n$ from $y$ and lower the entries of values $k+1,\ldots, n-1$ by $1$ respectively in $y_1\cdots y_{n-2}$. We then have an element in $S_{n-2}$, say, $y_1'\cdots y_{n-2}'$, and can apply the induction hypothesis to it. Therefore, $\pi_{\downarrow}(\mathsf{Pop}_{S_{n-2}}(y_1'\cdots y_{n-2}'))$ ends with $n-2$ and has no double descents. Now we take this image and add $1$ to entries of values $k,\ldots,n-2$ and denote it as $x_1'\cdots x_{n-2}'$. Because of the previous paragraph we have shown that $\pi_{\downarrow}(\mathsf{Pop}_{S_n}(y))= x_1'\cdots x_{n-2}'\cdot kn$, and the entire string has no double descents.

Now suppose $r\le n-2$. First we consider the case $y_{r-1}<y_{r+1}$. We have $\mathsf{Pop}_{S_n}(y)=\mathsf{Pop}_{S_{n-1}}(y_1\cdots y_{r-1}y_{r+1}\cdots y_n)n.$ Therefore, \begin{equation*}\begin{aligned}
\pi_{\downarrow}(\mathsf{Pop}_{S_n}(y))&=\pi_{\downarrow}\big(\mathsf{Pop}_{S_{n-1}}(y_1\cdots y_{r-1}y_{r+1}\cdots y_n)\cdot n\big)\\&=\pi_{\downarrow}\big(\mathsf{Pop}_{S_{n-1}}(y_1\cdots y_{r-1}y_{r+1}\cdots y_n)\big)\cdot n,\end{aligned}\end{equation*} where $\cdot$ stands for concatenation. We apply the induction hypothesis to $y_1\cdots y_{r-1}y_{r+1}\cdots y_n$, an element of $S_{n-1}$, and obtain that the first $n-1$ places of $x$ must not have consecutive double descents. Concatenating with $n$ will not change this statement, and we conclude this case.

Now we suppose $y_{r-1}>y_{r+1}$. Let $y_{q}y_{q+1}\cdots y_{r-1}$ be the longest descending run that ends with $y_{r-1}$. On one hand, $$\mathsf{Pop}_{S_n}(y_1\cdots  y_{r-1}ny_{r+1}\cdots y_n)=\mathsf{Pop}(y_1\cdots y_{q-1})\cdot y_{r-1}\cdots y_{q}y_n\cdots y_{r+1}n,$$ where $y_n<\cdots <y_{r+1}<y_{r-1}<\cdots < y_q$.

Now we start applying the series of swaps to apply $\pi_{\downarrow}$. Notice that every swap removes a $\overline{31}2$ pattern and $y_qy_ny_{r+1}$ is one such pattern. Thus, first $y_q$ is swapped with $y_n$. Then, $y_qy_{n-1}y_{r+1}$ should also be removed, so $y_q$ is again swapped with $y_{n-1}$. We repeat the process, and after $n-r$ swaps involving $y_q$ as the $c$ in $XcaYbZ$, the permutation becomes $$\mathsf{Pop}(y_1\cdots y_{q-1})\cdot y_{r-1}\cdots y_{q+1}y_n\cdots y_{r+1}y_qn.$$ Similarly, $y_{q+1}$ is moved to the end of $y_n\cdots y_{r+1}$, right before $y_qn$, and so is $y_{q+2},\ldots, y_{r-1}$. We arrive at $$\mathsf{Pop}(y_1\cdots y_{q-1})\cdot y_n\cdots y_{r+1}y_{r-1}\cdots y_{q}n.$$ We should clarify that the process of swapping is not finished yet; what we claim is that since $\pi_{\downarrow}$ is the same for sylvester-adjacent elements, we have $$\pi_{\downarrow}(\mathsf{Pop}_{S_n}(y))=\pi_{\downarrow}\big(\mathsf{Pop}(y_1\cdots y_{q-1})\cdot y_n\cdots y_{r+1}y_{r-1}\cdots y_{q}n\big).$$

On the other hand, $$\mathsf{Pop}_{S_n}(y_1\cdots y_{r-1}y_{r+1}\cdots y_n\cdot n)=\mathsf{Pop}(y_1\cdots y_{q-1})\cdot y_{n}\cdots y_{r+1}y_{r-1}\cdots y_q\cdot n.$$

Combining these observations we obtain that \begin{equation*}
    \begin{aligned}
    \pi_{\downarrow}(\mathsf{Pop}_{S_n}(y))&=\pi_{\downarrow}\big(\mathsf{Pop}_{S_n}(y_1\cdots y_{r-1}y_{r+1}\cdots y_n\cdot n)\big)\\&=\pi_{\downarrow}\big(\mathsf{Pop}_{S_{n-1}}(y_1\cdots y_{r-1}y_{r+1}\cdots y_n)\big)\cdot n.
    \end{aligned}
\end{equation*}
We apply the induction hypothesis to $y_1\cdots y_{r-1}y_{r+1}\cdots y_n$, an element of $S_{n-1}$, and obtain that the first $n-1$ places of $x$ must not have consecutive double descents. Concatenating with $n$ will not change this statement, and we conclude this case as well.

For the ``if'' direction, we suppose that $x=x_1\cdots x_n\in S_n$ with $x_n=n$ and $x$ has no consecutive double descents. We want to show that there is some 312-avoiding permutation $y$ such that $\pi_{\downarrow}(\mathsf{Pop}_{S_n}(y))=x$. We use strong induction on $x$'s length. 

We consider the position of $1$, say $x_k=1$. Then there are two immediate observations. Firstly, all entries $x_1,\ldots, x_{k-1}$ are smaller than all of $x_{k+1},\ldots, x_n$ to avoid a 312 pattern $x_jx_kx_{\ell}$ where $j<k<\ell $. Hence, it is clear that $\{x_1,\ldots, x_{k-1}\}=\{2,\ldots, k\}$ and $\{x_{k+1},\ldots, x_n\}=\{k+1,\ldots, n\}$. Secondly, if $k\ge 2$, then $x_{k-1}=k$. Otherwise, if $x_j=k$ for some other $j\le k-2$, then $x_jx_{j+1}x_{j+2}$ forms either a double descents or a 312-pattern, which is impossible.

We let $x_i'=x_i-1$ if $1\le i\le k-1$ and let $x_i'=x_i-k$ if $k+1\le i\le n$. Then $x_1'x_2'\cdots x_{k-1}'\in S_{k-1}$ and $x_{k+1}'x_{k+2}'\cdots x_n'\in S_{n-k}$ are two strings with no double descents, and $x_{k-1}'=k-1$, $x_n'=n-k$. Both of them satisfy the induction hypothesis, so we can find $z=z_1\cdots z_{k-1}\in S_{k-1}$ and $w=w_1\cdots w_{n-k}\in S_{n-k}$ such that $\pi_{\downarrow}(\mathsf{Pop}_{S_{k-1}}(z))=x_1'x_2'\cdots x_{k-1}'$ and $\pi_{\downarrow}(\mathsf{Pop}_{S_{n-k}}(w))=x_{k+1}'x_{k+2}'\cdots x_n'$.

Let $z'=z_1'\cdots z_{k-1}'$ where $z_i'=z_i+1$. Suppose $w_t=k+1$. Let $w'=w_1'\cdots w_t'\cdot 1\cdot w_{t+1}'\cdots w_{n-k}'$, where we let $w_i'=w_i+k$. Consider $y=z'\cdot w'$. It is clear that $y$ is 312-avoiding. Indeed, $z'$ and $w'$ are both 312-avoiding, and no pattern can be formed by entries from both segments because no entry of $z'$ can be larger than any entry of $w'$ except $1$. It suffices to show that $\pi_{\downarrow}(\mathsf{Pop}_{S_n}(y))=x$.

We carefully investigate $\pi_{\downarrow}(\mathsf{Pop}(w'))$ as follows. After $\mathsf{Pop}$, $w_t=k+1$ will be after $1$, and thus for $\pi_{\downarrow}$ we can perform a series of $XcaYbZ\to XacYbZ$ swaps with $a=1$ and $b=k+1$, until $1$ is perturbed to the start of this string. In other words, due to sylvester-adjacent elements have the same $\pi_{\downarrow}$ image, $$\pi_{\downarrow}(\mathsf{Pop}(w')=\pi_{\downarrow}(1\cdot \mathsf{Pop}(w_1'\cdots w_{n-k}'))=1\cdot \pi_{\downarrow}(\mathsf{Pop}(w_1'\cdots w_{n-k}')).$$

Since no pattern can be cross-composed by entries from both $z'$ and $w'$, we have that \begin{equation*}
\begin{aligned}
    \pi_{\downarrow}(\mathsf{Pop}_{S_n}(y))&=\pi_{\downarrow}(\mathsf{Pop}(z'))\cdot \pi_{\downarrow}(\mathsf{Pop}(w'))\\&=\pi_{\downarrow}(\mathsf{Pop}(z'))\cdot 1\cdot \pi_{\downarrow}(\mathsf{Pop}(w_1'\cdots w_{n-k}'))\\&=x_1\cdots x_{k-1}\cdot 1 \cdot x_{k+1}\cdots x_n,
    \end{aligned}
\end{equation*}
which is exactly $x$. This concludes the proof.

\end{proof}
\pagebreak
The last ingredient that we will need in the proof of \cref{Motzkin} is the following enumerative result.
\begin{theorem}$($\cite{Petersen}$)$\label{descentequalpeak}
The number of 231-avoiding permutations $\pi\in S_{n+1}$ with exactly $k$ descents and $k$ peaks is $\frac{1}{k+1}\binom{2k}{k}\binom{n}{2k}$. 
\end{theorem}
\begin{proof}[Proof of \cref{Motzkin}]
Define the bijective map $r(\pi)=\pi'=\pi_1'\cdots \pi_{n+1}'$ where $\pi_i'=n+2-\pi_{n+2-i}$. We claim that the effect of $r$ preserves the number of ascents (descents) of the permutation. Indeed, place $i$ being an ascent (descent) in $\pi'$ is equivalent to place $n+1-i$ being an ascent (descent) in $\pi$, respectively. Furthermore, if in $\pi$ the descending runs are of lengths $\ell_1,\ldots, \ell_m$, then  in $\pi'$ the descending runs are of lengths $\ell_m,\ldots, \ell_1$.

By \cref{descentequalpeak} it suffices for us to establish a bijection between 231-avoiding permutations $\pi\in S_{n+1}$ with exactly $k$ descents and $k$ peaks and $\{r(\pi)\mid \pi\in \mathsf{Pop}_{\mathrm{Av}_n(312)}(\mathrm{Av}_n(312)),\mathscr{U}_L(\pi)\\ =n-k\}$. On one hand, take $\pi$ from the former set and we have $\mathscr{U}_L(\pi')=n-k$, as having $k$ descents is equivalent to having $n-k$ ascents for elements in $S_{n+1}.$ Here, we use the well-known fact that $\mathscr{U}_L(\pi)$ equals to the number of ascents in $\pi$.

On the other hand, we will show that if $\mathscr{U}_L(\pi)=n-k$, then $r(\pi)=\pi'$ is 231-avoiding and has exactly $k$ descents and $k$ peaks. Being 231-avoiding and having $k$ descents are clear. Moreover, \cref{charMotzkin} establishes that $\pi$ has no double descents and ends with $n+1$. Therefore, $\pi'$ has no double descents either. This implies that the number of peaks of $\pi'$ is either equal to or is smaller by $1$ than the number of its descents, depending on whether the first index is a descent. Since $\pi'_{n+1}=n+2-\pi_{n+1}=1$, we know that $\pi'$ has $k$ peaks. This concludes the proof.
\end{proof}

\end{document}